\documentclass{article}      
\usepackage{amssymb}
\usepackage{amsmath}
\usepackage{latexsym}
\usepackage[mathscr]{euscript}
\usepackage{graphicx}
\usepackage{amscd}
\usepackage{amsthm} 
\usepackage{fullpage} 
\usepackage{wrapfig}
\newtheorem{theorem}{Theorem}
\newtheorem{corollary}[theorem]{Corollary}
\newtheorem{lemma}[theorem]{Lemma}  
\newtheorem{proposition}[theorem]{Proposition}

\newtheorem{definition}{Definition}

\newcommand{\bz}{\mathbb{Z}}

\newcommand{\bs}{\mathbb{S}}

\newcommand{\ce}{\mathcal{E}}

\newcommand{\cl}{\mathcal{L}}
\newcommand{\cs}{\mathcal{S}}
\newcommand{\cg}{\mathcal{G}}

\newcommand{\cp}{\mathcal{P}}

\newcommand{\hk}{\hookrightarrow}

\newcommand{\med}{\medskip}

\newcommand{\bfl}{\begin{flushleft}}
\newcommand{\efl}{\end{flushleft}}

\newcommand{\hocolim}{\operatorname{hocolim}}

\newcommand{\xr}{\xrightarrow}

\newcommand{\ltm}{LM^{-TM}}

 \newcommand{\haut}{hAut^R}
 \newcommand{\elx}{End^R_X \cl}
 \newcommand{\kg}{K(\Sigma^\infty (G_+))}

\begin{document}

  \title{Homotopy automorphisms of $R$-module bundles, and the $K$-theory of string topology}  
  \author{Ralph L. Cohen \thanks{The first author was partially supported by a  grant  from the NSF.} \\ Department of Mathematics \\Stanford University \\ Bldg. 380 \\ Stanford, CA 94305, USA \and  John D.S Jones \\ Mathematics Institute \\Zeeman Building \\ Warwick University\\ Coventry, CV4 7AL, UK }
\date{\today}
\maketitle  
 \begin{abstract}   Let $R$ be a ring spectrum and $ \ce \to X$   an $R$-module bundle of rank $n$.  
 Our main result  is to identify the homotopy type of the group-like monoid of homotopy automorphisms of this bundle, $hAut^R(\ce)$.   This will generalize the result regarding $R$-line bundles  proven by the authors in \cite{cjgauge}.  The main application is the calculation of the homotopy type of $BGL_n(End ((\cl))$ where $\cl \to X$ is any $R$-line bundle, and $End (\cl)$ is the ring spectrum of endomorphisms.   In the case when such a bundle is the fiberwise suspension spectrum of a principal bundle over a manifold, $G \to P \to M$,   this leads to a description of the $K$-theory of the string topology spectrum in terms of the mapping space from $M$ to $BGL (\Sigma^\infty (G_+))$.   \end{abstract}

 \tableofcontents

 \section*{Introduction}       Let $R$ be a ring spectrum.  In several places in the recent literature, the notion of an $R$-module bundle $\ce \to X$ of rank $n$  has been defined and described \cite{units}, \cite{5author}, \cite{lind}. This is a parameterized $R$-module spectrum $\ce$ over $X$, where each fiber $E_x$ admits an $R$-module equivalence $E_x \xr{\simeq} \vee_{n} R$.   In analogy to vector bundles, it was proved in \cite{lind} that equivalence classes of rank $n$ $R$-module bundles over $X$ are in bijective correspondence with the set of homotopy classes, $[X, BGL_n(R)] = \pi_0(Map(X, BGL_n(R))$.  
 
 The main theorem in this paper is the identification of the homotopy type of the group-like monoid of homotopy automorphisms, 
 $\haut (\ce)$.  
 This is the space of self equivalences  of $\ce$ living over the identity of $X$ that preserve the $R$-module structure.   A precise definition will be given in the text of the paper.
 
 \begin{theorem}\label{main} Let $R$ be a ring spectrum and $X$ a connected space of the homotopy type of a $CW$-complex.  Let  $\ce \to X$ be an $R$-module bundle of rank $n$.  Then  there is an equivalence of   group-like monoids,
   $$
   hAut^R(\ce) \simeq \Omega Map_\ce (X, BGL_n(R)) 
   $$
   where the subscript in this mapping space refers to the path component  of maps that classify $\ce$.  
 \end{theorem}

 \med
 The special case of this theorem when $n=1$ was proved by the authors in  \cite{cjgauge}.  As discussed there, this result is important in string topology.  Namely, given a principal bundle over a manifold $G \to P \to M$,  if we let 
$$\cl = \Sigma^\infty_M (P_+)$$
be the fiberwise suspension spectrum of $P$ with a fiberwise disjoint basepoint, then the string topology spectrum of $P$, $\cs (P) =  P^{-TM}$ is equivalent, as ring spectra, to the endomorphism ring $End^{\Sigma^\infty (G_+)}(\Sigma^\infty_M (P_+))$.  Thus if $R = \Sigma^\infty (G_+)$, and $\ce = \Sigma^\infty_M (P_+)$, then the above theorem, in the case  $n= 1$, describes the homotopy type of the group-like monoid of units, $GL_1(\cs (P))$.   
 
\med
  Theorem \ref{main} in its general setting will  have  following implication to string topology.

  \med
   Let $\cl \to X$ be an $R$-line bundle, and let $\oplus_n \cl \to X$ be the Whitney-sum of $n$-copies of $\cl$.  This is an $R$-module bundle of rank $n$.  
   
\begin{corollary}\label{bgln}  There is a homotopy equivalence
$$
BGL_n(\elx) \simeq Map_{\oplus_n\cl}(X, BGL_nR).
$$
\end{corollary}

As a special case
 we obtain the following result about the general linear groups of the string topology spectrum.

\med
\begin{corollary}\label{string} If $G \to P \to M$ is a principal bundle over a manifold and $\cl = \Sigma^\infty_M (P_+)$, there is a homotopy equivalence
$$
BGL_n (\cs (P)) \simeq Map_{\oplus_n\cl}(M, BGL_n(\Sigma^\infty(G_+)).
$$
   In particular there is an equivalence,
$$
BGL_n(LM^{-TM}) \simeq Map_{\iota_n}(M, BGL_n(\Sigma^\infty(\Omega M_+)).
$$
Here $\iota_n$ classifies $\oplus_n \cl$, where $\cl = \Sigma^\infty_M (\cp_+)$, and $\cp \to M$ is a universal bundle in the sense that $\cp$ is contractible. 
\end{corollary}

 Our next main result  describes  how these results have $K$-theoretic consequences for these ring spectra.  
 First note that 
  the string topology spectrum $\cs (P)$ is a \sl nonconnective \rm ring spectrum.  Indeed its homology groups are nontrivial through dimension $-n$, where $n$ is the dimension of the manifold $M$.   If $\cs$ is a ring spectrum,  let $K_{conn} (S)$ denote the   algebraic $K$-theory spectrum of its connective cover    $K(\cs_0)$.  This $K$-theory spectrum has zero-space $\Omega^\infty K_{conn} (\cs ) =  K_0(\pi_0(\cs)) \times BGL(\cs)^+$, where the superscript $+$ denotes a group completion that will be described in the text of the paper.     We will show that Corollary \ref{string} implies the following result about $K$-theory.

\begin{theorem}\label{ktheory}   Given an $R$-line bundle $\cl \to M$,  there is a homology isomorphism
$$
\alpha :  Map_{\cl} (M, BGL(\Sigma^\infty (G_+))) \to \Omega_0^\infty K_{conn}(\cs (P)). $$
 The subscript $0$ denotes  the path component of the basepoint in 
$\Omega^\infty K_{conn}(\cs (P))$.
  $Map_{\cl} (M, BGL(\Sigma^\infty (G_+))$ is the homotopy colimit of the mapping spaces $Map_{\oplus_n \cl} (M, BGL_n(\Sigma^\infty (G_+))$. \end{theorem}

\med
This theorem can be viewed as a statement about the  group completion of the above mapping spaces.   However we point out that this  is \sl not \rm the same as the mapping space to the group completion, which would be the zero space of the mapping spectrum $Map_0 (M, K(\Sigma^\infty (G_+))$.  This spectrum    calculates the $\kg$-cohomology of $M$.   
   However, as we will show below,  we can define  a homomorphism of $K$-theory groups,

\begin{equation}\label{kgcoho}
\gamma: K_{conn}^{-q}(\cs (P)) \to \kg^{-q}(M)
\end{equation}
which gives a partial geometric  understanding of the $\kg$-cohomology theory in terms of the algebraic $K$-theory of the string topology spectrum.  The situation when $q=0$ was studied in detail by Lind in \cite{lind}.  

We conclude by observing     two important applications of Theorem \ref{ktheory}.  

\med
\begin{corollary}\label{application} Let $M$ be a closed manifold. There are homology equivalences
\begin{align}
  Map_{\Sigma^\infty_M (\cp_+)} (M, BGL(\Sigma^\infty (\Omega M_+))  &\to \Omega^\infty_0 K_{conn}(LM^{-TM})   \notag  \\
  Map_{\bs} (M, BGL(\bs))   &\to    \Omega_0^\infty K_{conn}(DM)  \notag
\end{align}
where $\bs$ is the sphere spectrum, and $\ltm$ is the Thom spectrum of the virtual bundle $-TM$ over $M$, pulled back over $LM$ via the map $e : LM \to M$ that evaluates a loop at the basepoint of the circle.  $DM$ denotes the Spanier-Whitehead dual of the manifold $M$, which is an $E_\infty$-ring spectrum.
\end{corollary}

\med
We point out that in these cases, the map $\gamma$ defined above (\ref{kgcoho})  gives homomorphisms
$$
\gamma : K_{conn}^{-q}(LM^{-TM})  \to A(M)^{-q}(M)    \quad \text{and} \quad \gamma : K_{conn}^{-q}(D(M)) \to A(point)^{-q}(M).
$$

The algebraic $K$-theory of nonconnective spectra was defined in terms of Waldhausen categories of modules by Blumberg and Mandell in \cite{blumbergmandell}.  When $X$ is simply connected,  they related the Waldhausen category defining  $K(D(X))$ to the Waldhausen category defining $A(X) = K(\Sigma^\infty (\Omega X_+))$.  It would be interesting to relate Corollary \ref{application} regarding the   $K_{conn}$-theory to their  results.   

\med
This paper is organized as follows.  In Section 1 we will prove Theorem \ref{main} and derive Corollaries \ref{bgln} and \ref{string}.   In Section 2 we  describe the $K$-theoretic implications of Thoerem \ref{main},  and in particular  we prove Theorem \ref{ktheory}.

   \section{Automorphisms of $R$-module bundles}
   
   Let $R$ be a ring spectrum.   
   Let $\ce \to X$ be an $R$-module bundle of rank $n$,  in the sense of Lind \cite{lind}.  This is a parameterized spectrum over $X$ in the sense of May and Sigurdsson \cite{maysigurd}, where for each $x \in X$, the fiber $\ce_x$ is an $R$-module spectrum of rank $n$.  We denote the category of such bundles by $R-mod_n (X)$.  Again, this category was defined in \cite{lind}. It was shown there that equivalence classes of such bundles are classified by   homotopy classes of maps $X \to BGL_n(R)$.  Fix a particular map $\gamma_\ce : X \to BGL_n(R)$ classifying $\ce$.  This choice defines a basepoint in the mapping space $\gamma_\ce \in Map_\ce(X, BGL_n(R))$.  The endomorphisms of $\ce$ in $R-mod_n(X)$ is a parameterized spectrum which we denote by $End_M^R(\ce) \to X$.  For every $x \in X$ this defines a fiber spectrum $End_M^R(\ce)_x$ which is equivalent to the ring of endomorphisms $End^R(\vee_n R)$.  $End_M^R(\ce) $ is a parameterized ring spectrum under composition. By taking a fibrant replacement if necessary, we can take sections to produce an ordinary spectrum 
   $$
   End^R(\ce) = \Gamma_M(End^R_M (\ce)).
   $$
   The parameterized ring structure on $End_M^R(\ce)$ defines a ring spectrum structure on $End^R(\ce)$.
   
   \med
   
   \begin{definition}\label{haut}  We define the group-like monoid $\haut(\ce)$ to be the   units of the ring spectrum of endomorphisms,
   $$
   \haut (\ce) = GL_1(End^R(\ce)).
   $$
   \end{definition}

   \med
   
   We are now ready to prove Theorem \ref{main}.
   
   \begin{proof}

Consider the fiber bundle of infinite loop spaces 
given by taking the zero spaces of the fibrant model of $End^R_M(\ce)$:
$$
\Omega^\infty (End^R(\vee_n R)) \to    \Omega^\infty_M End^R_M(\ce) \to M.
$$
  By restricting to path components of those $R$-module endomorphisms that consist of equivalences, we get a subbundle, which we will call $ \cg L_n(\ce)$:
$$
GL_n(R) \to  \cg L_n(\ce)  \to M.
$$
 Notice that  $\haut(\ce)$ can be described as the space of sections of this bundle, $ \haut (\ce) = \Gamma_M(\cg L_n(\ce))$. We now observe that the fiber homotopy type of the bundle $\cg L_n (\ce) \to M$ has another description.  The homotopy class of map $\gamma_\ce : M \to BGL_n(R)$ that classifies the $R$-module bundle $\ce$,  also classifies a principal $GL_n(R)$-bundle over $M$:  
 $$
 GL_n(R) \to P_\ce \to M.
 $$
 Here we are replacing the group-like topological monoid $GL_n(R)$ by a topological group, which by abuse of notation, we continue to refer to as $GL_n(R)$.  As was shown in \cite{lind}, the relationship between $P_\ce$ and the parameterized spectrum $\ce$ is that there is an equivalence,  $$\Omega^\infty_M \ce  \simeq  P_\ce \times_{GL_n(R)} \Omega^\infty (\vee_n R),$$ where we are continuing the abuse of notation to allow  $\Omega^\infty (\vee_n R)$ to refer to a weakly homotopy equivalent  infinite loop space  that carries an action of the group $GL_nR$.   (See  \cite{maysigurd} and \cite{lind} for details.)    
 
 Now consider the corresponding adjoint bundle, $GL_n(R) \to P_\ce^{Ad} \to X$.  In this notation $P_\ce^{Ad} $ is the homotopy orbit space $P_\ce^{Ad} = P_\ce \times_{GL_n(R)} GL_n(R)$ where $GL_n(R)$ is acting on itself by the adjoint action (conjugation).  A standard observation about adjoint bundles implies that $P_\ce^{Ad} $   is the bundle whose fiber over $x \in M$ is the space  of $GL_n(R)$-equivariant automorphisms of  the fiber of $P_\ce$ at $x$.   The section space $\Gamma_M(P_\ce^{Ad})$ is therefore the group of equivariant automorphisms of $P_\ce$ living over the identity on $M$.  This is known as the gauge group, $\cg (P_\ce)$.
 
  From this viewpoint it becomes clear that there is a map of fibrations   $ P_\ce^{Ad} \to \cg L_n(\ce)$ over $M$ (after taking appropriate fibrant and cofibrant replacements), which is an equivalence on the fibers.     Therefore there is an equivalence of  their spaces of sections, as group-like $A_\infty$-spaces.
\begin{equation} \label{autgl}
 \phi:  \cg (P_\ce) = \Gamma_M( P_\ce^{Ad})  \simeq   \Gamma_M(\cg L_1(\ce))  =  hAut^R_X(\ce).
\end{equation}
 Now a well known     theorem of Atiyah and Bott \cite{atiyahbott}  says that the classifying space of the gauge group $\cg (P)$  of a principal bundle $G \to P \to M$ is  equivalent to the mapping space,
 $B\cg (P) \simeq Map_P(X, BG)$.  Theorem \ref{main} now follows by applying this 
  Atiyah-Bott equivalence   to the principal bundle $GL_n(R) \to P_\ce \to X.$
 \end{proof}

   \med
   We now consider an application of Theorem \ref{main} to an important special case. 
  Let $\cl \to X$ be an $R$-line bundle, and let $\oplus_n \cl \to X$ be the Whitney-sum of $n$-copies of $\cl$.  This is an $R$-module bundle of rank $n$, which is defined to be the pullback under the diagonal map $\Delta^n : X \to X^n$ of the exterior $n$-fold product $\cl^n \to X^n$.   Consider the endomorphism spectrum $End^R(\cl)$. As noted above, this is a ring spectrum.   We first need the following observation:

\med
\begin{lemma}\label{GLN}  There is an equivalence of group-like monoids
$$
\lambda : \haut (\oplus_n \cl) \simeq  GL_n(End^R(\cl)).  
$$
\end{lemma}

\begin{proof}
Given a ring spectum $S$,  recall that $GL_n(S)$ is defined to be the group-like monoid of units in the endomorphism ring,
$$
GL_n(S) =  GL_1(End^S(\vee_n S)).
$$
 Notice that there is a natural equivalence $End^S(\vee_n S) = \prod_n (\vee_n S)$.   Also notice that   $\vee_n End^R(\cl) \simeq Mor^R(\cl, \oplus_n \cl)$, where
 $Mor^R$ refers to the $R$-module morphisms in the category of parameterized spectra over $X$.  Thus when $S = End^R(\cl)$, we have a natural equivalence
 $$
 End^S(\vee_n S) \simeq \prod_n(Mor^R(\cl, \oplus_n \cl) = End^R (\oplus_n \cl).
 $$
 Furthermore this equivalence clearly preserves the ring structure.  It therefore  induces an equivalence  of their group-like monoids of units,
 $$
 GL_1(End^S(\vee_n S)) \simeq GL_1(End^R(\oplus_n \cl).
 $$
 The left side is by definition $GL_n (End^R(\cl))$, and the right side is by definition $\haut (\oplus_n \cl)$.
\end{proof}

We now notice that  Theorem \ref{main} and Lemma \ref{GLN} together imply the following.  

\begin{corollary}\label{bgln}  There is a homotopy equivalence
$$
BGL_n(\elx) \simeq Map_{\oplus_n\cl}(X, BGL_nR).
$$
\end{corollary}

When $G \to P \to M$ is a principal bundle over a manifold and $\cl = \Sigma^\infty_M (P_+)$, we obtain the following result about the general linear groups of the string topology spectrum,

\med
\begin{corollary}\label{string} There is a homotopy equivalence
$$
\beta : BGL_n (\cs (P)) \xr{\simeq} Map_{\oplus_n \cl}(M, BGL_n(\Sigma^\infty(G_+))
$$
where     $\cl$ is the line bundle $\Sigma^\infty (G_+) \to \Sigma^\infty_M(P_+) \to M$.    In particular there is an equivalence 
$$
BGL_n(LM^{-TM}) \simeq Map_{\iota_n}(M, BGL_n(\Sigma^\infty(\Omega M_+)).
$$
\end{corollary}

\section{$K$-theoretic implications}
  
  The goal of this section is to use  Corollary \ref{string} to prove Theorem \ref{ktheory} as stated in the introduction.  We will then use it to derive descriptions
  of the $K$-theory of the connective covers of the Spanier-Whitehead dual $D(M) = Map (\Sigma^\infty (M_+), \bs)$ and of the string topology ring spectrum $\ltm$. 
  
  To do this we need to understand the equivalence given in Corollary \ref{bgln} more carefully, so that we can deduce $K$-theoretic consequences.  As in that corollary, let $\cl \to M$ be an $R$-line bundle and let $P_\cl \to M$ denote the principal $GL_1R$-bundle defined by the classifying map of $\cl$,  $\gamma_\cl : M \to BGL_1(R)$.  Let $\oplus_n P_\cl$ denote the  principal bundle classified by $$M \xr{\Delta} \prod_n M  \xr{\gamma_\cl^n } \prod_n BGL_1(R) \xr{\mu} BGL_n(R).$$ Here $\mu$ is the usual block addition.  This is the principal $GL_n(R)$-bundle associated to the $n$-fold Whitney sum $\oplus_n \cl$. 
  
  Combining  equivalence (\ref{autgl}) with Lemma \ref{GLN}  defines an equivalence of  group-like $A_\infty$-spaces,
  $$
  \begin{CD}
 \Psi_n :  \cg (\oplus_n  P_\cl)   @>\phi > \simeq >   hAut^R(\oplus_n \cl) @>\lambda >\simeq >  GL_n(End^R(\cl)).
  \end{CD}
  $$
   
   \med
   Notice that these equivalences respect the  standard inclusions of wreath products, that makes the following diagrams commute:
   
   \med
   
\begin{equation}\label{wreath}
   \begin{CD}
   \Sigma_k \int \cg(\oplus_n P_\cl)      @>\hk >>   \cg ((\oplus_{nk} P_\cl) \\
   @V 1 \times \Psi_n^k VV     @VV\Psi_{nk} V \\
   \Sigma \int (GL_n(End^R(\cl)))  @>>\hk >  GL_{nk}(End^R(\cl)).
   \end{CD}
   \end{equation}

   \med
   Using the Barratt-Eccles $E_\infty$-operad,  we can conclude the following:
   
   \med
   \begin{proposition}
   The disjoint unions $\coprod_{n \geq 0} B\cg(\oplus_n P_\cl)$ and $\coprod_{n\geq 0} BGL_n(End^R(\cl))$
   have $E_\infty$-algebra structures, and the equivalence
   $$
   \sqcup  B\Psi_n  :  \coprod_{n \geq 0} B\cg(\oplus_n P_\cl)  \to  \coprod_{n\geq 0} BGL_n(End^R(\cl))
   $$
   is an equivalence of $E_\infty$-spaces.  
    \end{proposition}

    \med
   We are now ready to deduce $K$-theoretic information from the above analysis.

   \med
  If $X$ is an $E_\infty$ space,  let $X^+$ denote the corresponding infinite loop space given by its group completion.
   In other words, 
   $$
   X^+ = \Omega B(X).
   $$

   \med
   
   Given a ring spectrum $S$, let $S_0$ denote its   connective  cover.  We denote by $K_{conn}(S)$ the algebraic $K$-theory spectrum  $K(S_0)$.  A connected component of the  zero space of this spectrum is given by  $\Omega_0^\infty K_{conn}(S) = \left(\coprod_n BGL_n(S)\right)_0^+$, the component of the basepoint in the group completion.    We note that the usual algebraic $K$-theory spectrum $K(S)$ is defined as the $K$-theory of the Waldhausen-category of perfect modules over $S$  \cite{blumbergmandell}.  When $S$ is not connective, these spectra are not necessarily equivalent, where they are for connective ring spectra.

   \med
   We can now conclude the following.  
   
   \med
   \begin{proposition}\label{firstK}  If $R$ is a ring spectrum and $R \to \cl \to M$ is an $R$-line bundle over a closed manifold $M$,  then each path component of the zero space of the $K$-theory spectrum of the endomorphism ring has the following homotopy type:
  $$
  \Omega_0^\infty (K_{conn}(End^R(\cl))) \simeq   \left(\coprod_{n\geq 0} BGL_n(End^R(\cl))\right)_0^+ \simeq  \left(\coprod_{n \geq 0} B\cg(\oplus_n P_\cl)\right)_0^+.
  $$
  \end{proposition}
  
  \med
  We now proceed with the proof of Theorem \ref{ktheory} as stated in the introduction.
  
  \med
  \begin{proof} As is usual, let $BGL (End (\cl))$ denote the homotopy colimit of the inclusion maps $BGL_n (End (\cl)) \hk BGL_{n+1} (End (\cl))$  induced by the usual inclusions of the groups $GL_n (End (\cl)) \hk GL_{n+1}(End(\cl))$.  Then by the group-completion theorem of McDuff-Segal \cite{mcduffsegal},  there is a homology equivalence,
   $$
   \alpha : \bz \times BGL (End(\cl)) \to \left(\coprod_{n\geq 0} BGL_n(End^R(\cl))\right)^+  \simeq \bz \times \Omega_0^\infty  (K_{conn}(End^R(\cl))).
   $$
   Furthermore, by Proposition \ref{firstK}  there is an equivalence
   $$
   \bz \times \hocolim_n B\cg(\oplus_n P_\cl) \simeq \bz \times BGL (End(\cl)).
   $$
 The maps in the colimit are induced by the inclusions of gauge groups, $j_n:  \cg(\oplus_n P_\cl) \hk \cg(\oplus_{n+1}P_\cl)$ given by thinking of an automorphism of $\oplus_n P$ as an automorphism of $\oplus_{n+1} P$ by taking the direct sum with the identity map on the last factor.
 
 \med
  Now for any principal bundle  $G \to P \to M$, consider the Atiyah-Bott equivalence \cite{atiyahbott} mentioned in the proof of Theorem \ref{main}:
  $$
  \beta : B\cg (P) \simeq  Map_P(M, BG).
  $$
  This equivalence is given by the observation that the equivariant mapping space $Map^G(P, EG)$ is a contractible space with a free action of the gauge
  group $\cg(P) = Aut^G_M(P)$ given by precomposition.  The orbit space of this action is then  $Map_P (M, BG)$.  The examples relevant here
  are the principal bundles $GL_n(R) \to \oplus_n P_\cl \to M$.  In this case  we have an   equivalence 
  $$
 \beta_n:  B\cg (\oplus_n P_\cl) \xr{\simeq}   Map_{\oplus_n\cl}(M, BGL_nR).
 $$
  With respect to these equivalences, the inclusions $j_n : B\cg (\oplus_n P_\cl) \to B\cg (\oplus_{n+1} P_\cl)$  are given up to homotopy by maps 
  $$
  q_n : Map_{\oplus_n\cl}(M, BGL_nR) \to Map_{\oplus_{n+1}\cl}(M, BGL_{n+1}R) 
  $$
  defined  by sending $\phi : M \to BGL_n(R)$ to the composition $$q_n(\phi) : M \xr{\Delta} M \times M  \xr{\phi \times \gamma_\cl }  BGL_n(R) \times BGL_1(R) \xr{\mu} BGL_{n+1}(R)$$  where $\gamma_\cl$ is the fixed basepoint in $Map_\cl (M,  BGL_1(R))$, and $\mu$ is the usual block pairing.

 If we denote the homotopy colimit of these maps $hocolim_n Map_{\oplus_n\cl}(M, BGL_nR)$ by   $Map_\cl (M, BGL(R))$,   we then have a homology equivalence
 $$
 \bz \times Map_\cl (M, BGL(R)) \xr{\simeq} \bz \times \hocolim_n B\cg(\oplus_n P_\cl) \simeq \bz \times BGL (End(\cl)) \to  \bz \times \Omega_0^\infty K_{conn}(End (\cl)),
 $$ which is the statement of Theorem \ref{ktheory}. 
   \end{proof}
   
\med
We now proceed to prove Corollary \ref{application},  describing  $K_{conn}$ of both the string topology spectrum, and of the Spanier-Whitehead dual of a manifold.     In the case of the string topology spectrum $\cs (M) = \ltm$,  this is the string topology of the universal principal bundle.  That is, a principal bundle over $M$ fiber homotopy  equivalent to the path-loop fibration  $\Omega M \to \cp_M \to M$.      In the case of the Spanier-Whitehead dual, $DM$,  one sees that if one considers the trivial principal bundle, when $G = \{id\}$,   then the string topology spectrum of this bundle
is simply the dual, $DM = Map(M, \bs)$. Both of these examples are analyzed in more detail in \cite{cjgauge}.     Corollary \ref{application}   follows by applying   Theorem \ref{ktheory} to these special cases.

 \end{document}